\documentclass[12pt]{amsart}

\usepackage{amsmath}
\usepackage{amssymb}
\usepackage{amsthm}
\usepackage{amscd}
\usepackage[all,cmtip]{xy}
\usepackage{graphics}

\newtheorem{theorem}{Theorem}
\newtheorem{lemma}[theorem]{Lemma}
\newtheorem{corollary}[theorem]{Corollary}

\newtheorem{question}{Question}
\newtheorem{property}[question]{Property}
\newtheorem*{remark}{Remark}

\DeclareMathOperator{\ord}{ord}
\begin{document}
\title{Around rational functions invertible in radicals}
\author{Y. Burda}

\begin{abstract}
A class of rational functions characterized by some wonderful properties is studied. The properties that identify this class include simple algebra (their inverses can be expressed in radicals), simple topology (the total space of the minimal Galois covering dominating them has genus 0 or 1) and simple local topology (branching data). Explicit formulae for these functions are obtained as well as their classification up to different equivalence relations.
\end{abstract}

\maketitle

\section{Introduction}

The central subject of this paper is a family of rational functions having some wonderful properties. These functions appeared, for instance, in the work \cite{Ritt} of J.F. Ritt, where the following question has been asked:

\begin{question} \label{inverse-radicals-question}
What are the rational functions $R:\mathbb{CP}^1\to\mathbb{CP}^1$ with inverses expressible in radicals (i.e. such that the equation $w=R(z)$ can be solved for $z$ in radicals)?
\end{question}

Ritt has identified these functions in two cases: when the degree of the rational function was prime and when the function was a polynomial. The results of Ritt were largely forgotten for some time and were partially rediscovered by A.G. Khovanskii in \cite{Khovanskii}.

The functions that appeared as an answer to Ritt's question \ref{inverse-radicals-question} in the case the degree was prime turn out to have a simple  uniform description --- they all have the following property:

\begin{property}\label{main-question}
The rational functions $R$ fit into a commutative diagram
$$\xymatrix{
A_1 \ar[r] \ar[d] & A_2 \ar[d]\\
S \ar[r]^R & S
}
$$
where the arrow between $A_1$ and $A_2$ is an isogeny of algebraic groups of dimension $1$, which are either $\mathbb{C}^*$ or elliptic curves, and the vertical arrows are quotients by the action of a finite group of automorphisms of the algebraic group $A_i$ ($S$ in this case is isomorphic either to $\mathbb{C}^*$, if $A_i$ is $\mathbb{C}^*$, or to $\mathbb{CP}^1$, if $A_i$ is an elliptic curve). 
\end{property}

Instead of the algebraic characterization above one can characterize the same family of functions (with some small additions) by the following topological property:

\begin{property} \label{total-space-property}
The total space of the minimal Galois branched covering that dominates the branched covering $R:\mathbb{CP}^1\to\mathbb{CP}^1$is non-hyperbolic (i.e. has genus 0 or 1).
\end{property}

This topological description can be formulated in local, rather than global terms as well. Namely one finds that the branching data of the rational functions that have property \ref{total-space-property} is rather simple (e.g. there are at most 4 branching points). One can in fact extract a characterization of these functions by their branching data. This is done explicitly in section \ref{branching-section}. Here we can instead characterize these functions implicitly as the answer to the following question:

\begin{question} \label{bounded-genus-question}
Let $w_1,\ldots,w_b\in \mathbb{CP}^1$ be a given collection of points and let $d_1,\ldots,d_b\in\mathbb{N}\cup \{0\}$ be some given natural numbers. Consider the collection of holomorphic functions $R$ from some compact Riemann surface to $\mathbb{CP}^1$ that are branched over the points $w_1,\ldots,w_b$ with the local monodromy $\sigma_i$ around the point $w_i$ satisfying $\sigma_i^{b_i}=1$. Suppose that the functions in this collection have bounded topological complexity in the following sense: the genus of their source Riemann surface is bounded. What are the rational functions in this collection?
\end{question}

Question \ref{bounded-genus-question} was of the form "when simple local topology implies a simple global topology for the mapping?". This question has an algebraic analogue of the form "when simple branching data implies simple algebraic properties for the mapping?":

\begin{question} \label{simple-algebra-question}
Let $w_1,\ldots,w_b\in \mathbb{CP}^1$ be a given collection of points and let $d_1,\ldots,d_b\in\mathbb{N}\cup \{0\}$ be some given natural numbers. Consider the collection of holomorphic functions $R$ from some compact Riemann surface to $\mathbb{CP}^1$ that are branched over the points $w_1,\ldots,w_b$ with the local monodromy $\sigma_i$ around the point $w_i$ satisfying $\sigma_i^{b_i}=1$. Suppose that the functions in this collection have inverses expressible in radicals. What are the rational functions in this collection?
\end{question}

The answer to this question is the same as the answer to the previous one, with the exception of icosahedral rational functions (case (2,3,5) below).

This paper is mainly concerned with the questions/properties outlined above. However the functions we will be discussing appear in arithmetic applications as well. Namely these functions appear as a part of the answer to the following question:

\begin{question}
What are the rational functions $R$ defined over a number field $K$ so that for infinitely many places the map induced by $R$ on the projective line defined over the corresponding residue field is a bijection?
\end{question}

This question (inspired by conjecture of Schur) has been answered in a works of M. Fried \cite{Fried} and R.Gularnick, P. Muller and J. Saxl \cite{GMS}  

I hope to devote a separate paper to the multidimensional generalization of these rational functions, stemming mainly from property \ref{main-question}. These analogues have been explored in the case where the functions are polynomials in works of Hoffmann and Withers \cite{HoffmannWithers} and Veselov \cite{Veselov}

The paper is structured as follows:

In section \ref{galois-coverings-section} we describe all Galois branched coverings over the Riemann sphere with total space of the covering having genus 0 or 1. In section \ref{branching-section} we give a necessary and sufficient condition on the local branching of a branched covering over the Riemann sphere to be dominated by a Galois covering from section \ref{galois-coverings-section}. Questions \ref{bounded-genus-question} and \ref{simple-algebra-question} is discussed in the end of section \ref{branching-implies-section}. In section \ref{ritt-problem-section} we apply the results of the previous section to prove Ritt's theorem more naturally than in \cite{Ritt}. In section \ref{classification-section} we improve Ritt's result to give three classifications of rational functions with inverses expressible in radicals: up to left-analytical equivalence, up to left-right analytical equivalence and up to orientation-preserving topological equivalence. In section \ref{formulae_section} we give explicit formulae for Ritt's rational functions. In section \ref{appendix-section} we formulate some standard group-theoretic results we needed in the paper. Finally in section \ref{funny-section} we give some funny small applications of results mentioned in this paper.

The author thanks his advisor A.G. Khovanskii for suggesting to look at Ritt's forgotten works, for many interesting and helpful discussions and for encouragement to think about the ideas discussed in this paper.

\section{Conventions} \label{conventions_section}

A branched covering over a compact Riemann surface $S$ is a non-constant holomorphic map $f:\tilde{S}\to S$ from a compact Riemann surface $\tilde{S}$ to $S$. For any point $z_0\in \tilde{S}$ there exist holomorphic coordinate charts around $z_0$ and $f(z_0)$ so that in these coordinates $f(z)=z^{d_{z_0}}$. The integer $d_{z_0}$ is called the multiplicity of $f$ at $z_0$.

The local monodromy at a point $w\in S$ is the conjugacy class inside the monodromy group of the permutation induced by going along a small loop in the base around $w$.

 A branched covering $f:\tilde{S}\to S$ is called Galois if there exists a group $G$ of automorphisms of $\tilde{S}$ that are compatible with $f$ (i.e. $f\circ g=f$ for any $g\in G$) that acts transitively on all fibers of $f$. For Galois branched coverings the multiplicities of all points in a fiber over a fixed point $w$ in $S$ are the same. This common multiplicity will be called the multiplicity of the Galois branched covering $f$ at $w$.

\section{Non-hyperbolic Galois branched coverings over the Riemann sphere} \label{galois-coverings-section}

In this section we will describe the Galois branched coverings over the Riemann sphere, whose total space has genus 0 or 1, i.e. is topologically a sphere or a torus. Equivalently it is a description of all non-free actions of a finite group on a sphere or a torus.

Let $f:T\to \mathbb{CP}^1$ be a Galois branched covering of degree $d$. Let $w_1,\ldots,w_b\in \mathbb{CP}^1$ be all the branching points of $f$ and let $d_{w_i}$ denote the local multiplicity at $w_i$.

Riemann-Hurwitz formula for $f$ states that $$\chi(T)=\chi(\mathbb{CP}^1)d-\sum_{i=1}^b{(d-\frac{d}{d_i})}$$ or $$b-2+\frac{\chi(T)}{d}=\frac{1}{d_1}+\ldots+\frac{1}{d_b}$$

If the surface $T$ is a sphere then $\chi(T)=2$ and the formula specializes to $$b-2+\frac{2}{d}=\frac{1}{d_1}+\ldots+\frac{1}{d_b}$$

If there are only two branching points, the local monodromies around them must be the same, since the product of small loops around the two branching points is contractible in the compliment to the branching locus. In this case we have then a family of solutions $$b=2;d_1=d_2=d$$

In the case $b=3$ we have the following solutions:

case (2,2,*): $$b=3,d_1=2,d_2=2,d_3=d$$

case (2,3,3): $$b=3,d_1=2,d_2=3,d_3=3;d=12$$

case (2,3,4): $$b=3,d_1=2,d_2=3,d_3=4;d=24$$

case (2,3,5): $$b=3,d_1=2,d_2=3,d_3=5;d=60$$

For $b\ge 4$ there are no solutions.

If the surface $T$ is the torus, then $\chi(T)=0$ and the formula specializes to $$b-2=\frac{1}{d_1}+\ldots+\frac{1}{d_b}$$

The solutions of this equation are

case (2,3,6): $$b=3,d_1=2,d_2=3,d_3=6$$

case (2,4,4): $$b=3,d_1=2,d_2=4,d_3=4$$

case (3,3,3): $$b=3,d_1=3,d_2=3,d_3=3$$

case (2,2,2,2): $$b=4,d_1=2,d_2=2,d_3=2,d_4=2$$

We see that there are only few possibilities for such Galois coverings and they all have very simple branching data. In section \ref{branching-section} we will see that this simple local picture characterizes these coverings uniquely.

\section{Auxiliary definitions and results} \label{auxiliary-section}

To prove the results of section \ref{branching-section} we will need the following two lemmas.

\begin{lemma} \label{two-branchings-lemma}
If a branched covering over $\mathbb{CP}^1$ is branched at most at two points, then the total space of the covering is the Riemann sphere. Moreover, if the degree of the covering is $d$, then a holomorphic coordinate on source and target spheres can be chosen so that the covering is given by $z\to z^d$.

If a covering over a genus one curve is unbranched, then the total space of the covering has genus one.
\end{lemma}

\begin{lemma} \label{local-branching-lemma}
\emph{(Local picture of pull-back)}
Let $D$ denote the unit disc in $\mathbb{C}$. Let $f:D\to D$ denote the branched covering $z\to z^n$ and $g:D\to D$ denote the branched covering $z\to z^m$. Let $\tilde{f},\tilde{g},X$ fit into the pullback diagram
$$\begin{CD}
X         @>\tilde{f}>> D\\
@VV\tilde{g}V           @VVgV\\
D         @>f>>          D
\end{CD}$$

Then $X$ is a disjoint union of $\gcd(m,n)$ discs, $\tilde{f}$ restricted to each of the discs is of the form $z\to z^{\frac{n}{\gcd(m,n)}}$, $\tilde{g}$ restricted to each of the discs is of the form $z\to z^{\frac{m}{\gcd(m,n)}}$.
\end{lemma}

For purposes of section \ref{branching-section} we will also need the Galois coverings $T\to T/G$, where the Riemann surface $T$ and the group $G$ acting on it are as following:

Case (2,2,*): $T=\mathbb{CP}^1$, $G=\mathbb{Z}_2$, $G$ acts by an involution fixing two points on $\mathbb{CP}^1$.

Case (2,3,3): $T=\mathbb{CP}^1$, $G=A_4$ acting on the sphere by rotations of a regular tetrahedron

Case (2,3,4): $T=\mathbb{CP}^1$, $G=S_4$ acting on the sphere by rotations of an octahedron

Case (2,3,5): $T=\mathbb{CP}^1$, $G=A_5$ acting on the sphere by rotations of an icosahedron

Case (2,4,4): $T=\mathbb{C}/\langle1,i\rangle$, $G=\mathbb{Z}_4$ with generator acting by $z\to i z$

Case (2,3,6): $T=\mathbb{C}/\langle1,\omega\rangle$, $G=\mathbb{Z}_6$ with generator acting by $z\to -\omega z$, $\omega^2+\omega+1=0$.

Case (3,3,3): $T=\mathbb{C}/\langle 1,\omega\rangle$, $G=\mathbb{Z}_3$ with generator acting by $z\to \omega z$, $\omega^2+\omega+1=0$.

Case (2,2,2,2): $T$ is any elliptic curve, $G=\mathbb{Z}_2$, with the nontrivial element acting by $z\to -z$.

In each case the quotient space $T/G$ can be identified with $\mathbb{CP}^1$. For any of the non-trivial subgroups $H$ of $G$ the quotient space by the action of $H$ can also be identified with $\mathbb{CP}^1$.

\section{Characterization by local branching} \label{branching-section}

In this section we will prove the following result:

\begin{theorem}\label{classification}
Let $f:S\to \mathbb{CP}^1$ be a branched covering from a Riemann surface $S$ to the Riemann sphere and suppose it is branched at $b$ points $w_1,\ldots,w_b$ with local monodromies $\sigma_1,\ldots,\sigma_b$ around them satisfying $\sigma_i^{d_i}=1$ for one of the following cases:

case (*,*): $b=2,d_1=d_2=0$ (i.e. we are given that it is branched at two points)

case (2,2,*): $b=2,d_1=d_2=2,d_3=0$ (i.e. we are given that it is branched at most at three points and local monodromies around two of them square to identity)
\begin{itemize}
\item case (2,3,3): $b=3,d_1=2,d_2=3,d_3=3$

\item case (2,3,4): $b=3,d_1=2,d_2=3,d_3=4$

\item case (2,3,5): $b=3,d_1=2,d_2=3,d_3=5$

\item case (2,3,6): $b=3,d_1=2,d_2=3,d_3=6$

\item case (2,4,4): $b=3,d_1=2,d_2=4,d_3=4$

\item case (3,3,3): $b=3,d_1=3,d_2=3,d_3=3$

\item case (2,2,2,2): $b=4,d_1=2,d_2=2,d_3=2,d_4=2$
\end{itemize}
Then the function $f$ fits into one of the following diagrams:

case (*,*): 
$$\xymatrix{
\mathbb{CP}^1 \ar[r]^{z\to z^d} \ar[d]^{\simeq} & \mathbb{CP}^1 \ar[d]^{\simeq}\\
S \ar[r]^f & \mathbb{CP}^1
}
$$

case (2,2,*):
$$\xymatrix{
\mathbb{CP}^1 \ar[r]^{z\to z^d} \ar[d] & \mathbb{CP}^1 \ar[d]\\
\mathbb{CP}^1/H  \ar[d]^{\simeq} \ar[r] & \mathbb{CP}^1/G \ar[d]^{\simeq} \\
S \ar[r]^{f} & \mathbb{CP}^1
}
$$
where $G=\mathbb{Z}_2$ and $H$ is one of its subgroups - either $\mathbb{Z}_2$ or the trivial group.

cases (2,3,3),(2,3,4),(2,3,5):
$$\xymatrix{
\mathbb{CP}^1 \ar[d] \ar[dr] &\\
\mathbb{CP}^1/H  \ar[d]^{\simeq} \ar[r] & \mathbb{CP}^1/G\ar[d]^{\simeq} \\
S \ar[r]^f & \mathbb{CP}^1
}
$$
where $G$ is the group for the corresponding case from section \ref{auxiliary-section} and $H$ is any of its subgroups.

cases (2,4,4),(2,3,6),(3,3,3),(2,2,2,2):
$$\xymatrix{
T_1 \ar[r]^{\tilde{f}} \ar[d] & T \ar[d]\\
T_1/H  \ar[d]^{\simeq} \ar[r] & T/G\ar[d]^{\simeq} \\
S \ar[r]^f & \mathbb{CP}^1
}
$$
where $T_1$ and $T$ are curves of genus $1$, $\tilde{f}$ is an unramified covering of tori, $G$ is the group for the corresponding case from section 2 and $H$ is one of its subgroups (the automorphism group of the curve $T_1$ should contain $H$).

In particular the Riemann surface $S$ is either the Riemann sphere or (in the cases (2,4,4),(2,3,6),(3,3,3),(2,2,2,2) when $H$ is the trivial group) a torus.

\end{theorem}

\begin{remark} By identifying the quotient spaces by the actions of the groups $G$ and $H$ with $\mathbb{CP}^1$ when possible (i.e. except the cases when $H$ is trivial and acts on a torus) and choosing an appropriate holomorphic coordinate on this $\mathbb{CP}^1$ we get very explicit description of the rational function $f$ (up to change of coordinate in the source and in the target):

In case (2,2,*) the quotient by action of $\mathbb{Z}_2$ can be realized in appropriate coordinates by $z\to\frac{z+z^{-1}}{2}$, so the rational function $f$ is the Chebyshev polynomial: $f(\frac{z+z^{-1}}{2})=\frac{z^d+z^{-d}}{2}$.

In case (2,2,2,2) the quotient by the action of $\mathbb{Z}_2$ can be realized by the Weierstrass $\wp$-function, and thus $f$ is the rational function that expresses the $\wp$-function of a lattice and in terms of the $\wp$-function of its sublattice.

In case (3,3,3) the quotient by the action of $\mathbb{Z}_3$ can be realized by $\wp'$.

In case (2,4,4) the quotient by the action of $\mathbb{Z}_4$ can be realized by $\wp''$.

In case (2,3,6) the quotient by the action of $\mathbb{Z}_6$ can be realized by $\wp^{(4)}$.
\end{remark}

\begin{proof}
In case (*,*) the claim is the content of lemma \ref{two-branchings-lemma}.

In case (2,2,*) identify $T/\mathbb{Z}_2$ (for $T=\mathbb{CP}^1$ and action of $\mathbb{Z}_2$ as in section  \ref{auxiliary-section}) with $\mathbb{CP}^1$ in such a way that the two branching points of the quotient map $q_G:\mathbb{CP}^1\to\mathbb{CP}^1/\mathbb{Z}_2\simeq\mathbb{CP}^1$ coincide with points $w_1$ and $w_2$.

In other cases choose the identification of $T/G$ from section \ref{auxiliary-section} with $\mathbb{CP}^1$ so that the branching points $w_i$ and the numbers $d_i$  attached to them (i.e the numbers so that the local monodromy to the power $d_i$ is the identity) coincide with the branching points for the quotient mapping  $q_G:T\to T/G\simeq\mathbb{CP}^1$ and their local multiplicities.
 
Consider the pullback of the map $q_G$ by the map $f$. Let the pullback maps be called as in the diagram below
$$\xymatrix{
\tilde{T} \ar[r]^{\tilde{f}} \ar[d]^{\tilde{q}_G} & T \ar[d]^{q_G}\\
S \ar[r]^f & \mathbb{CP}^1
}
$$

The map $\tilde{q}_G$ realizes the quotient map by the action of $G$ on a possibly reducible curve $\tilde{T}$.

Let $H$ be the subgroup of $G$ fixing some component of $\tilde{T}$ and let $T_1$ denote this component. Then we can restrict the diagram to this component:
$$\xymatrix{
T_1 \ar[r]^{\tilde{f}} \ar[d]^{q_H} & T \ar[d]^{q_G}\\
S \ar[r]^f & \mathbb{CP}^1
}
$$

(The map $q_H$ is the restriction of $\tilde{q}_G$ to the connected component $T_1$; it realizes the quotient map by the action of $H$).

Now we can use lemma \ref{local-branching-lemma} to investigate how the map $\tilde{f}$ is branched.

For every point $w_i$ in $\mathbb{CP}^1$ for which we know that the local $f$-monodromy $\sigma$ satisfies $\sigma^{d_i}=1$ with $d_i>0$ we know also that the local monodromy of $q_G$ at this point is a disjoint unit of cycles of length $d_i$. Hence $\tilde{f}$ is unramified over any $q_G$-preimage of such points.

In case (2,2,*) this means that $\tilde{f}$ is branched over at most at two points --- the two preimages of the point $w_3$ under $q_G$.

In other cases this means that $\tilde{f}$ is unramified.

In case (2,2,*) lemma \ref{two-branchings-lemma} implies that $T_1$ is the Riemann sphere and by changing identification of $T$ and $T_1$ with $\mathbb{CP}^1$ and choosing an appropriate holomorphic coordinate, we can make it be given by $z\to z^d$.

Hence $f$ fits into the diagram
$$\xymatrix{
\mathbb{CP}^1 \ar[r]^{z\to z^d} \ar[d] & \mathbb{CP}^1 \ar[d]\\
\mathbb{CP}^1/H  \ar[d]^{\simeq} \ar[r] & \mathbb{CP}^1/G \ar[d]^{\simeq} \\
S \ar[r]^{f} & \mathbb{CP}^1
}
$$
as required.

In cases (2,3,3),(2,3,4),(2,3,5) the covering $\tilde{f}$ is unramified covering over the Riemann sphere, hence $\tilde{f}$ is the identity map.

Thus in this case $f$ fits into the diagram
$$\xymatrix{
\mathbb{CP}^1 \ar[d] \ar[dr] &\\
\mathbb{CP}^1/H  \ar[d]^{\simeq} \ar[r] & \mathbb{CP}^1/G\ar[d]^{\simeq} \\
S \ar[r]^f & \mathbb{CP}^1
}
$$

In cases (2,4,4),(2,3,6),(3,3,3),(2,2,2,2) the covering $\tilde{f}$ is an unramified covering of a torus over a torus. In the diagram 
$$\xymatrix{
T_1 \ar[r]^{\tilde{f}} \ar[dd]^{q_H} & T \ar[d]\\
  & T/G\ar[d]^{\simeq} \\
S \ar[r]^f & \mathbb{CP}^1
}
$$
the map $q_H$ must be one of the quotient maps by the action of a finite group of automorphisms of the genus 1 curve $T_1$ as described in section \ref{auxiliary-section}, because the action of $H$ (if $H$ is not trivial) has fixed points and all the actions with this property appeared in section \ref{auxiliary-section}. In this case $S$ is either a torus (if $H$ is trivial) or the Riemann sphere (if it is not).

Finally $f$ fits into the diagram
$$\xymatrix{
T_1 \ar[r]^{\tilde{f}} \ar[d] & T \ar[d]\\
T_1/H  \ar[d]^{\simeq} \ar[r] & T/G\ar[d]^{\simeq} \\
S \ar[r]^f & \mathbb{CP}^1
}
$$
as required.
\end{proof}

\section{When branching data implies global properties} \label{branching-implies-section}
In this section we will answer questions \ref{bounded-genus-question} and \ref{simple-algebra-question}. The answers to these questions explain why the branching data that appear in the formulation of theorem \ref{classification} are in some sense optimal. The result we will be proving is the following.

\begin{theorem}
Let $w_1,\ldots,w_b\in \mathbb{CP}^1$ be a given collection of points and let $d_1,\ldots,d_b\in\mathbb{N}\cup \{0\}$ be some given natural numbers. Consider the collection of holomorphic functions $R$ from some compact Riemann surface to $\mathbb{CP}^1$ that are branched over the points $w_1,\ldots,w_b$ with the local monodromy $\sigma_i$ around the point $w_i$ satisfying $\sigma_i^{b_i}=1$.

The genus of the source Riemann surface for all functions in this collection is bounded if and only if the numbers $(d_1,\ldots,d_b)$ are the ones that appear in theorem \ref{classification}.

All the functions in this collection have inverses expressible in radicals if and only if the numbers $(d_1,\ldots,d_b)$ are the ones that appear in theorem \ref{classification}, with the exception of case $(2,3,5)$.
\end{theorem}

\begin{proof}
For the proof we make two observations: the first one is that if the function $R:S\to \mathbb{CP}^1$ belongs to this collection, then the minimal branched Galois covering $\tilde{R}:\tilde{S}\to\mathbb{CP}^1$ that dominates $R$ also belongs to it. The second observation is that if $S_1\to S$ is an unramified covering and $R:S\to \mathbb{CP}^1$ is a function from the collection, then the composite function $S_1\to S\to \mathbb{CP}^1$ also belongs to the collection.

Now suppose that for given $w_1,\ldots,w_b$ and $d_1,\ldots,d_b$ as above, the collection contains some function that is not one of those mentioned in theorem \ref{classification}. Then it also contains its "Galois closure" - the minimal Galois branched covering that dominates it, and the source space of this Galois closure has genus at least 2. Hence we can find an unramified covering over it with total space of arbitrary high genus, contradicting the assumption of the first claim. Also we can find an unramified covering over it with unsolvable monodromy group, so that the inverse to this mapping can't be expressed in radicals, contradicting the assumptions of the second claim.

Finally for the branching data appearing in theorem \ref{classification} the genus of the source Riemann surface is either 0 or 1. Similarly the possible monodromy groups for all the branching data from theorem \ref{classification} are all solvable, with the exception of the group $A_5$ in case (2,3,5).

\end{proof}
 
\section{Applications to Ritt's problem} \label{ritt-problem-section}

In work \cite{Ritt} Ritt was interested in rational functions whose inverses were expressible in radicals. Over the field of complex numbers this problem is equivalent to describing branched coverings of the Riemann sphere over itself with solvable monodromy group.

If we are interested in rational functions that can't be expressed as compositions of other rational functions, we should only consider those branched coverings of $\mathbb{CP}^1$ over $\mathbb{CP}^1$ with solvable primitive monodromy group (see section \ref{appendix-section}).

Result of Galois tells that a primitive action of a solvable group on a finite set can be identified with an irreducible action of a subgroup of the group of affine motions of the vector space $F_p^n$.

In case $n=1$ above, i.e. when the degree of the rational function is a prime number, the action can be identified with action of a subgroup of $F_p^*\rtimes F_p$ on $F_p$.

Now let $w_1,\ldots,w_b$ be the branching points and let the local monodromies around these points be given by $x\rightarrow a_i x+b_i \mod p$ in the above identification. A simple application of Riemann-Hurwitz formula shows \[\sum{\frac{1}{\ord a_i}}=b-2\] with the convention that $\ord 1=\infty$. 

This equation has the following solutions for $\ord a_i$:

1. $\frac{1}{\infty}+\frac{1}{\infty}=2-2$

2. $\frac{1}{2}+\frac{1}{2}+\frac{1}{\infty}=3-2$

3. $\frac{1}{2}+\frac{1}{3}+\frac{1}{6}=3-2$

4. $\frac{1}{2}+\frac{1}{4}+\frac{1}{4}=3-2$

5. $\frac{1}{3}+\frac{1}{3}+\frac{1}{3}=3-2$

6. $\frac{1}{2}+\frac{1}{2}+\frac{1}{2}+\frac{1}{2}=4-2$

It is easy to check that if $\ord a_i \neq \infty$ then the $\ord a_i$th power of the permutation $x\rightarrow a_i x+b_i$ is the identity.

Thus we are exactly in the situation described in theorem \ref{classification} and the results thereof apply here.

Moreover, since the degree of our covering is prime, degree considerations force the subgroup $H$ from theorem \ref{classification} to be equal to the group $G$.

Thus we get immediately that we can change coordinates in the source and the target $\mathbb{CP}^1$ so that our function $f$ becomes either the map $z\rightarrow z^p$ or fits into one of the following diagrams
$$\xymatrix{
\mathbb{CP}^1 \ar[r]^{z\rightarrow z^p} \ar[d]_{\frac{z+1/z}{2} \gets z} & \mathbb{CP}^1 \ar[d]^{z\rightarrow \frac{z+1/z}{2}}\\
\mathbb{CP}^1 \ar[r]^f & \mathbb{CP}^1
}
$$

$$\xymatrix{
T_1 \ar[r]^{\tilde{f}} \ar[d]^{q_G} & T_2 \ar[d]^{q_G}\\
\mathbb{CP}^1 \ar[r]^f & \mathbb{CP}^1
}
$$
where $T_1$ and $T_2$ are tori on which the same group $G$ acts with fixed points, $\tilde{f}$ is an unbranched covering of tori, the vertical arrows are quotients by the action of $G$ and $G$ is $\mathbb{Z}_2$, $\mathbb{Z}_3$, $\mathbb{Z}_4$ or $\mathbb{Z}_6$.

\section{Three classifications of Ritt's functions} \label{classification-section}

Here we would like to obtain slightly more precise results, that show how many rational functions invertible in radicals are there up to equivalence. We will consider three equivalence relations: left analytical equivalence ($f_1$ and $f_2$ are equivalent if there exists a Mobius transformation $\mathbb{CP}^1\stackrel{\sim}{\to} \mathbb{CP}^1$ so that the diagram
$$\xymatrix{
\mathbb{CP}^1 \ar[r]^{\sim} \ar[dr]^{f_1} & \mathbb{CP}^1 \ar[d]^{f_2}\\
  & \mathbb{CP}^1
}
$$
commutes), left-right analytical equivalence ($f_1$ and $f_2$ are equivalent if there exist two Mobius transformations $\mathbb{CP}^1\stackrel{\sim}{\to} \mathbb{CP}^1$ so that the diagram
$$\xymatrix{
\mathbb{CP}^1 \ar[r]^{\sim} \ar[d]^{f_1} & \mathbb{CP}^1 \ar[d]^{f_2}\\
\mathbb{CP}^1  \ar[r]^{\sim} & \mathbb{CP}^1
}
$$
commutes), and topological equivalence (which is defined by the same diagram, but instead of Mobius transformations we have orientation-preserving homeomorphisms of the sphere).

We start with the easiest of these, the left-analytical equivalence.

We will consider the case (2,2,2,2) in some details and only state the result for other cases.

Let $w_1,w_2,w_3,w_4$ be four branching points in $\mathbb{CP}^1$. We want to count how many rational functions of prime degree $p$ are there up to left analytical equivalence with branching points $w_1,w_2,w_3,w_4$ so that their inverses are expressible in radicals.

Consider the two-sheeted branched covering $q:T\rightarrow \mathbb{CP}^1$ over $\mathbb{CP}^1$ branched over $w_1,w_2,w_3,w_4$. Let $O\in T$ denote the preimage of $w_1$. We can identify the space $T$ with an elliptic curve with zero at $O$.

Now we claim that the set of left equivalence classes of degree $p$ isogenies from some elliptic curve to $T$ is in bijection with the set of left analytical equivalence classes of rational functions we are interested in.

The bijection is constructed in the following way.

Take the left equivalence class of an isogeny $s:T_1\rightarrow T$. The composition $q\circ s$ satisfies $q\circ s(z)=q\circ s (-z)$ for every $z\in T_1$ (because $s(-z)=-s(z)$ and $q(-s(z))=q(s(z))$).

Hence it factors through $T_1/z\sim -z$. Choose an isomorphism of $T_1/z\sim -z$ with $\mathbb{CP}^1$. The bijection sends the class of $s$ to the class of the resulting map $\tilde{s}:\mathbb{CP}^1\cong T_1/z\sim -z\rightarrow\mathbb{CP}^1$ fitting into the diagram
$$\xymatrix{
T_1  \ar[d] \ar[rr]^{s} && T \ar[d]^{q}\\
T_1/z\sim -z &\cong\mathbb{CP}^1 \ar[r]^{\tilde{s}} & \mathbb{CP}^1
}
$$

It is easy to check that this map is well-defined.

The inverse of this map can be defined in the following way. Consider left equivalence class of rational mapping $g:\mathbb{CP}^1\rightarrow \mathbb{CP}^1$ of degree $p$ invertible in radicals and branched over $w_1,\ldots,w_4$.

Galois lemma tells us that the squares of local monodromies around points $w_1,\ldots,w_4$ are identity.

Result of section \ref{branching-section} tells us that in the  pullback diagram
$$\xymatrix{
T_1 \ar[r]^{\tilde{g}} \ar[d]^{\tilde{q}} & T \ar[d]^{q}\\
\mathbb{CP}^1 \ar[r]^g & \mathbb{CP}^1
}
$$

the space $T_1$ is a torus, $\tilde{g}$ is an unramified covering of tori. Hence if we choose the origin in $T_1$ to be one of the preimages of $O$, we get that $\tilde{g}$ is an isogeny of degree $p$ over $T$.

We send the class of $g$ to the class of $\tilde{g}$. This map is also well-defined and is the inverse of the map we had defined previously.

Hence all we should do is count the equivalence classes of degree $p$ isogenies over an elliptic curve $T$. These are enumerated by the subgroups of $\pi_1(T,O)$ of index $p$. Since $\pi_1(T,O)$ is isomorphic to $\mathbb{Z}^2$, the number we are interested in is $p+1$: index $p$ subgroups of $\mathbb{Z}^2$ are in bijection with index $p$ subgroups of $\mathbb{Z}_p^2$, i.e. the points of $\mathbb{P}^1(\mathbb{F}_p)$.

Next we will construct the space that parametrizes left-right equivalence classes of rational mappings of prime degree with four branching points and inverses expressible in radicals.

First we claim that instead of parameterizing such left-right classes of rational mappings, we can parametrize isogenies of genus one Riemann surfaces with marked quadruples of points which are the fixed points of an involution with fixed points (another way to say it is that the four distinct marked points $w_1,...,w_4$ satisfy $2w_i\sim 2w_j$, where $\sim$ stands for linear equivalence)

Indeed, to any rational function $g:\mathbb{CP}^1\rightarrow \mathbb{CP}^1$ branched at four points with the property that the square of monodromy around each one of them is the identity, we can associate the isogeny $\tilde{g}:T_1\rightarrow T$ of elliptic curves that fits into the pullback diagram 
$$\xymatrix{
T_1 \ar[r]^{\tilde{g}} \ar[d]^{\tilde{q}} & T \ar[d]^{q}\\
\mathbb{CP}^1 \ar[r]^g & \mathbb{CP}^1
}
$$
where $q$ is the double covering over $\mathbb{CP}^1$ with four branching points that coincide with the branching points of $g$. The mapping $\tilde{q}$ must be branched over the four regular preimages under $g$ of the branching points of $g$. The curves $T_1$ and $T$ have genus 1 and are equipped with involutions with fixed points - the involutions that interchange the two sheets of the covering $q$ and of the covering $\tilde{q}$.

Two rational functions $g_1$ and $g_2$ are left-right equivalent if and only if they give rise to isomorphic isogenies $\tilde{g_1}$ and $\tilde{g_2}$ (in the sense that the isogenies are isomorphic and marked points get carried to marked points by the isomorphisms).

Now we can choose origin on the curve $T$ to be at one of the marked points and the origin of $T_1$ to be at its preimage under $\tilde{g}$. This way we get an isogeny of elliptic curves of prime degree. Different choices of origin give rise to isomorphic isogenies of elliptic curves (composition with translations provide the isomorphisms) and hence we have to parametrize the space of isogenies of degree $p$ of elliptic curves.

This space is known (see for instance \cite{Shimura}) to be the modular curve that is the quotient of the upper half-plane $\mathbb{H}$ by the action of the group $\{\left(\begin{matrix}a & b \\ c & d\end{matrix}\right) \in PSL_2(\mathbb{Z})|c\equiv 0 \mod p\}$. This curve admits a degree $p+1$ branched covering over the moduli space of elliptic curves (the covering map sends the isogeny to its target elliptic curve). This covering is branched over two points --- the classes of elliptic curves $\mathbb{C}/\langle1,i\rangle$ and $\mathbb{C}/\langle1,\omega\rangle$.

Thus for a generic choice of 4 branching points for the rational map of degree $p$ with inverse invertible in radicals there are $p+1$ left-equivalence classes of such rational functions and all of them are not left-right equivalent to each other. For the choices of these branching points with the property that they are either harmonic or have cross ratio $e^{\pm i \pi /3}$, among the $p+1$ left-equivalence classes there are left-right equivalent ones. Namely for the case that the branching points are harmonic, the $p+1$ left-equivalence classes get partitioned by the left-right analytic equivalence relation to $\frac{p-1}{2}$ pairs and 2 singletons if $p\equiv 1 \mod 4$ and just to $\frac{p+1}{2}$ pairs if $p\equiv 3\mod 4$ (if $p=2$, the partition is to one pair and one singleton). For the case of cross ratio  $e^{\pm i \pi /3}$, the $p+1$  left-equivalence classes get partitioned by the left-right analytic equivalence relation to $\frac{p-1}{3}$ triples and 2 singletons if $p\equiv 1 \mod 6$ and just to $\frac{p+1}{3}$ triples if $p\equiv 5\mod 6$ (if $p=3$, the partition is to one triple and one singleton, if $p=2$, the partition is to one triple).

Finally there is only one class of such mappings up to left-right topological equivalence, since the modular curve is connected.

In other cases the picture is even simpler, since we don't have any moduli in question.

Namely, for the case $(2,4,4)$, to count left-equivalence classes we should count the number of sublattices of index $p$ of the lattice of Gaussian integers, which are invariant under multiplication by $i$. These exist only if $p=2$ or $p\equiv 1 \mod 4$, and then, if $p=2$, there is only one such, and if $p\equiv 1 \mod 4$, there are two such. These classes of left-equivalence are not left-right-analytically equivalent, and hence also not topologically equivalent with orientation preserving homeomorphisms (they are the connected components of the corresponding Hurwitz scheme).

Similarly in cases $(2,3,6)$ and $(3,3,3)$ we should count the number of sublattices of the lattice of Eisenstein integers, which are invariant under multiplication by the primitive cubic root of unity $\omega$. These exist only if $p=3$ or $p\equiv 1 \mod 6$. If $p=3$ there is one such, and in case $p=1\mod 6$, there are two such.

\section{Explicit formulae for the rational functions invertible in radicals} \label{formulae_section}

In the previous sections we've seen that the non-polynomial rational functions of prime degree with the inverse expressible in radicals are in fact the same as the rational functions that express some particular elliptic function on a lattice in terms of an elliptic function on its sublattice. The elliptic functions that appeared were the functions that generated the field of elliptic functions on the lattice, invariant under a certain group of automorphisms of the elliptic curve. For instance in the $(2,2,2,2)$ case, the corresponding rational function should have expressed the Weierstrass function of a lattice in terms of the Weierstrass function for a sublattice.

Since we would like to derive the expressions for the corresponding rational functions uniformly for cases $(2,4,4)$,$(3,3,3)$,$(2,3,6)$ and $(2,2,2,2)$, we are going to use the following notations:

$S_\Lambda(z)$ denotes
\begin{itemize}
\item{$\wp_\Lambda(z)$ in case $(2,2,2,2)$}

\item{$\wp'_\Lambda(z)$ in case $(3,3,3)$}

\item{$\wp''_\Lambda(z)$ in case $(2,4,4)$}

\item{$\wp^{(4)}_\Lambda(z)$ in case $(2,3,6)$ }
\end{itemize}
Also $W$ will denote the group of roots of unity of degree
\begin{itemize}
\item{$2$ in case $(2,2,2,2)$}

\item{$3$ in case $(3,3,3)$}

\item{$4$ in case $(2,4,4)$}

\item{$6$ in cases $(2,3,6)$}
\end{itemize}

We will denote by $g[0]$ the constant term in a Laurent expansion of the meromorphic function $g$ at $0$.

We will need to use the following formula relating the Weierstrass $\wp$-function with the Jacobi theta function $\theta$:
$$\log(\theta(z))''=\wp(z)+c$$
where $c$ is a constant.

Now we are ready:

Let $\Lambda$ be a sublattice of index $p$ in a lattice $\Lambda'$ and assume that both $\Lambda$ and $\Lambda'$ are invariant under multiplication by elements of $W$. We want to find a rational function $R$ such that $S_{\Lambda'}(z)=R(S_\Lambda(z))$.

Since the functions $S_{\Lambda'}$ and $\sum_{u\in \Lambda'/{\Lambda}}{S_\Lambda(z-u)}$ have the same poles and are both periodic with respect to $\Lambda'$, we can write $$S_{\Lambda'}(z)-S_{\Lambda'}[0]=\sum_{0\neq u\in \Lambda'/{\Lambda}}{(S_\Lambda(z-u)-S_\Lambda(-u))} +S_\Lambda(z) - S_{\Lambda}[0]$$

We are now going to group the summands to $W$-orbits and find a formula for $\sum_{\xi\in W}{(S_\Lambda(z-\xi u)-S_\Lambda(- \xi u))}$ in terms of $S_\Lambda(z)$.

To do so, let $f(u)=\prod_{\xi\in W}{\frac{\theta(z-\xi u)}{\theta(-\xi u)}}$. The function $f$ is periodic with respect to $\Lambda$, has simple zeroes at $u=\xi z$ and a pole of order $|W|$ at $0$. Hence $f(u)=C (S_\Lambda(u)-S_\Lambda(z))$ for some constant $C$.

Now we can take the logarithm of both sides and differentiate the result with respect to $u$ (note that the differentiation is with respect to $u$, not $z$!) $|W|$ times. The result is the beautiful formula

$$(-1)^{|W|}\sum_{\xi\in W} {(S_\Lambda(z-\xi u)-S_\Lambda(-\xi u))}=\frac{\partial^{|W|}}{\partial{u}}{\log{(S_\Lambda(u)-S_\Lambda(z))}}$$

Note that the right hand side is a rational function of $S_\Lambda(z)$.

Hence we get 
\begin{align*} S_{\Lambda'}(z)-S_{\Lambda'}[0] =&\sum_{0\neq u \in (\Lambda'/\Lambda)/W}{\sum_{\xi\in W}{(S_{\Lambda}(z-\xi u)-S_{\Lambda}(-\xi u))}}&+S_{\Lambda}(z)-S_{\Lambda}[0] \\ = (-1)^{|W|}&\sum_{0\neq u \in (\Lambda'/\Lambda)/W}\frac{\partial^{|W|}}{\partial{u}}{\log{(S_\Lambda(u)-S_\Lambda(z))}} &+S_{\Lambda}(z)-S_{\Lambda}[0]\end{align*}

So finally $$R(w)=(-1)^{|W|}\sum_{0\neq u \in (\Lambda'/\Lambda)/W}\frac{\partial^{|W|}}{\partial{u}}{\log{(S_\Lambda(u)-w)}} +w +S_{\Lambda'}[0] -S_{\Lambda}[0]$$

For instance in the case $(2,2,2,2)$, with lattice $\Lambda'$ generated by $1$ and $\tau$ and the lattice $\Lambda$ generated by $p,\tau$, the function $R$ is $$R(w)=\sum_{u=1}^{\frac{p-1}{2}}{\left(\frac{\wp''(u)}{\wp(u)-w}-\left(\frac{\wp'(u)}{\wp(u)-w}\right)^2\right)}+w$$

\section{Auxiliary results used in the paper} \label{appendix-section}

One reason for an algebraic function to be expressible in radicals is that it is a composition of two such functions. Thus one can be interested in theorems that characterize this situation. The following two theorems are of this kind:

\begin{theorem} Let $f:(X,x_0)\rightarrow (Z,z_0)$ be a covering map between two (connected, locally simply connected) pointed spaces and let $M_f:\pi_1(Z,z_0)\rightarrow S(f^{-1}(z_0))$ be the monodromy mapping. The covering $f$ can be decomposed as a composition of two coverings $g:(X,x_0)\rightarrow (Y,y_0)$ ang $h:(Y,y_0)\rightarrow (Z,z_0)$ if and only if the monodromy group $M_f(\pi_1(Z,z_0))$ acts imprimitively on $f^{-1}(z_0)$
\end{theorem}

\begin{theorem}Let $G\leq S(X)$ be a solvable primitive group of permutations of a finite set $X$. Then the set $X$ can be identified with the $F_p$-vector space $F_p^n$ for some prime number $p$ and number $n\geq 1$ in such a way that the group $G$ gets identified with a subgroup of affine motions of the vector space $F_p^n$ that contains all translations. 
\end{theorem}

\begin{corollary}
Let $f:(X,x_0)\rightarrow (Z,z_0)$ be a covering map with a solvable group of monodromy. Then it can be decomposed as a composition of covering maps $(X,x_0)\xrightarrow{f_0}(Y^{(1)},y_0^{(1)})\xrightarrow{f_0}(Y^{(2)},y_0^{(2)})\rightarrow\ldots\xrightarrow{f_{k-1}}(Y^{(k)},y_0^{(k)})\xrightarrow{f_k}(Z,z_0)$ so that each covering $f_i$ has degree $p_i^{n_i}$ for some prime number $p_i$ and has primitive group of monodromy.
\end{corollary}

We will start by proving the first theorem in the list.

\begin{proof}
For one direction suppose that a covering $(X,x_0)\xrightarrow{f}(Z,z_0)$ can be decomposed as a composition of coverings $(X,x_0)\xrightarrow{g}(Y,y_0)\xrightarrow{h}(Z,z_0)$. Then the fiber $f^{-1}(z_0)$ can be decomposed into blocks which consist of preimages under $g$ of the points in the fiber $h^{-1}(z_0)$. These blocks get permuted among themselves by any loop in $(Z,z_0)$, thus showing that the monodromy group of the covering $f$ is imprimitive.
Conversely, let the monodromy $M_f$ act imprimitively on the fiber $f^{-1}(z_0)$. Let $y_0,...,y_n$ denote a system of blocks that shows that the action is imprimitive (each $y_i$ is a block). We number them so that the block $y_0$ contains the point $x_0$. Then the monodromy action of $\pi_1(Z,z_0)$ on $f^{-1}(z_0)$   gives rise to its action on the set of imprimitivity blocks ${y_0,\ldots,y_n}$. Denote by $(Y,y_0)\xrightarrow{h}(Z,z_0)$ the covering that corresponds to the subgroup of $\pi_1(Z,z_0)$ that stabilizes the block $y_0$. We call the chosen point $y_0$ by the same name as the block $y_0$, which shouldn't cause confusion (one can think of $g^{-1}(z_0)$ as the set of blocks ${y_0,\ldots,y_n}$). It remains to show that the map $f:(X,x_0)\rightarrow (Z,z_0)$ factors through $h:(Y,y_0)\rightarrow (Z,z_0)$. By the theorem of lifting of coverings it is enough to show that $f_*(\pi_1(X,x_0))$ is contained in $g_*(\pi_1(Y,y_0))$. The first of these groups is equal to the subgroup of $\pi_1(Z,z_0)$ that stabilizes $x_0$, while the second is equal (by construction) to the subgroup of $\pi_1(Z,z_0)$ that stabilizes the block $y_0$. From the definition of imprimitivity blocks, if a loop stabilizes $x_0$, then it must stabilize the block which contains it, namely $y_0$. This observation finishes the proof that imprimitivity of monodromy action implies that the covering $f$ decomposes.
\end{proof}

Now we will prove the second theorem, which is contained in the works of Galois.

\begin{proof}
Let $G\geq G^{(1)} \geq G^{(2)} \geq \ldots \geq G^{(k-1)} \geq G^{(k)}=1$ be the derived series for the group $G$ (i.e. $G^{(i+1)}$ is the commutator subgroup of $G^{(i)}$). Each of the subgroups $G^{(i)}$ is normal in the group $G$. Hence the group $G^{(k-1)}$ is an abelian normal subgroup of $G$. Let $p$ be any prime dividing the order of $G^{(k-1)}$ and let $N\leq G^{(k-1)}$ be the subgroup consisting of elements of order $p$ in it and the identity element: $N={n\in G^{(k-1)} : n^p=1}$. It is a subgroup, because $G^{(k-1)}$ is abelian. It is a characteristic subgroup of $G^{(k-1)}$, hence it is normal in $G$. Since the order of every non-trivial element of $N$ is $p$, the abelian group $N$ can be identified with the $F_p$-vector space $F_p^n$ for some natural number $n$.
Since the group $G$ acts primitively on $X$, the action of the normal subgroup $N$ must be transitive (because the orbits under action of any normal subgroup of $G$ form imprimitivity blocks for the action of $G$). Now since $N$ is abelian, the action must be regular as well (because the stabilizers of points in $X$ under the action of a transitive group must all be conjugate to each other, but in an abelian group this means that the stabilizers of all the points are equal). Thus every element in $N$ that fixes some point in $X$ must fix all points in $X$, hence it is the trivial permutation of the set $X$ and thus must be trivial itself.
Since the action of $N$ on $X$ is free and transitive, the points in $X$ can be identified with elements of $N$. Namely, choose $x_0\in X$ be any point. We will identify any other point $x$ with the unique element $n\in N$ that sends $x_0$ to $x$.
Let now $G_{x_0}$ be the stabilizer subgroup of $x_0$ in $G$. We claim that every element $g$ of $G$ can be written uniquely as a product $nh$ with $n\in N$ and $h\in G_{x_0}$. Indeed, let $n$ be the unique element of $N$ that maps $x_0$ to $g\cdot x_0$. Then the element $n^{-1} g$ stabilizes $x_0$, i.e. belongs to $G_{x_0}$. Denote it by $h$. Thus $g=n h$ with $n\in N$ and $h\in G_{x_0}$. If $g=n h=n' h'$ are two representations of $g$ in such form, then $n'^{-1}n=h' h^{-1}$ must be an element of $N$ that fixes $x_0$. But we have proved that it must then be the identity element.
Finally, after identifying the points of $X$ with the elements of the $F_p$-vector space $N$, the elements of $G_{x_0}$ act by linear mappings. Indeed, if we denote by star the action of $G_{x_0}$ on $N$ coming from identification of $X$ with $N$, then $g*n$ should be the (unique) element of $N$ that sends $x_0$ to $gnx_0$. Since the element $gng^{-1}$ belongs to $N$ and sends $x_0$ to $gnx_0$, we are forced to declare that $g*n=gng^{-1}$. Now the fact that $G_{x_0}$ acts linearly on $N$ is evident.
Thus we can identify the set $X$ with the vector space $F_p^n$, the abelian group $N$ with the group of translations of this vector-space and the group $G_{x_0}$ with the group of linear transformation of it.
\end{proof}

\section{Related results} \label{funny-section}

We can apply our classification results for instance to the classification of rational functions of degree 3. Indeed, when the degree of the rational function is 3, only the following branching data is possible: 4 double points, 2 double points and one triple point, or 2 triple points. In the first case the function is of the kind we described as case (2,2,2,2) and thus we get the following classification result: for every four-tuple of points in $\mathbb{CP}^1$ there are 4 left equivalence classes of rational functions of degree 3 with four simple branching points. If we look at the moduli space of left-right equivalence classes of such functions, we get that this space admits a degree 4 branched covering over the moduli space of unordered four-tuples of points in $\mathbb{CP}^1$, i.e. over the space $\mathbb{C}$ --- the coordinate on this space being the $j$-invariant (which is defined in terms of the cross ratio $\tau$ of these points as $\frac{4}{27}\frac{(\tau^2-\tau+1)^3}{\tau^2(\tau-1)^2}$). It is branched over points 0 and 1 --- with a triple branching point over 0 and two double points over 1. It is natural to compactify this branched covering by adding a point $\infty$ in the target, corresponding to the case when two branhing points get merged together, and two points in the source --- one for the case when the rational function stays a rational function in the limit (with two double branching points and one triple) and one for the case when the rational function degenerates to a mapping from a reducible curve with two irreducible components to the sphere, with the first of these being a triple point for the branched covering. This description of course agrees well with the desription of this moduli space as the modular curve $X_0(3)$.

Another small application is a characterization of polynomials, invertible in radicals. The theorems in section \ref{appendix-section} tell that a polynomial of degree 6 is invertible in radicals if and only if it is a composition of polynomials of degrees $p^k$ for some prime numbers $p$ invertible in radicals having primitive monodromy groups. For prime degee we have already shown that these polynomials are (up to composition with linear functions) either the power functions $z\to z^p$, or the Chebyshev polynomial. In his work \cite{Ritt} Ritt has shown that if a polynomial of degree $p^k$ with $k>1$ has primitive solvable monodromy group, then $p=2,k=2$, meaning that the polynomial is of degree $4$. This allows an absolutely explicit characterization of polynomials invertible in radicals --- any such polynomial is a composition of linear functions, power functions $z\to z^p$, Chebyshev polynomials and polynomials of degree $4$.

To see why any such polynomial is in fact invertible in radicals, we should show that power functions, Chebyshev polynomials and degree 4 polynomials are all invertible in radicals. For power functions this is trivial. Chebyshev polynomials can be inverted in radicals explicitly: to solve the equation $$w=T_n(z)$$ where $T_n$ is the degree $n$ Chebyshev polynomial, one can use the following trick. We know that Chebyshev polynomial is defined by the property that $T_n(\frac{u+1/u}{2})=\frac{u^n+u^{-n}}{2}$. We can write $w$ as $\frac{u^n+u^{-n}}{2}$ --- for this we should take $u=\sqrt[n]{w+\sqrt{w^2-1}}$. Then $$z=\frac{u+1/u}{2}=\frac12\left(\sqrt[n]{w+\sqrt{w^2-1}}+\sqrt[n]{w-\sqrt{w^2-1}}\right)$$

Finally there are many ways to solve degree 4 polynomial equations in radicals, and hence also to invert a degree 4 polynomial in radicals (see for instance the discussion in \cite{Khovanskii} and references therein).

\end{document}